\documentclass[leqno,12pt]{article}

\usepackage{amsmath}
\usepackage{amscd}
\usepackage{amsopn}
\usepackage{amsthm}
\usepackage{amsfonts,amssymb}
\usepackage{amsfonts,bbm}
\usepackage{latexsym}

\newtheorem{theorem-nonumber}{THEOREM}

\newtheorem{lemma}{LEMMA} %  [section]
\newtheorem{proposition}[lemma]{PROPOSITION}
\newtheorem{corollary}[lemma]{COROLLARY}
\newtheorem{theorem}[lemma]{THEOREM}
\newtheorem{remark}[lemma]{REMARK}

\newtheorem{example}[lemma]{EXAMPLE}

\newcommand{\real}{\mathbbm{R}}
\newcommand{\nat}{\mathbbm{N}}

\newcommand{\limn}{\lim_{n \to \infty}}

\newcommand{\ve}{\varepsilon}

\newcommand{\reald}{{\real^d}}
\newcommand{\realt}{{\real^2}}

\newcommand{\und}{\quad\mbox{ and }\quad}

\newcommand{\ov}{\overline}

\newcommand{\C}{\mathcal C}  

\newcommand{\F}{\mathcal F}

\renewcommand{\S}{\mathcal S}
\newcommand{\M}{\mathcal M}

\newcommand{\dist}{\mbox{\rm dist}}

\newcommand{\supp}{\operatorname*{supp}}

\newcommand{\itemframe}%
{\setlength{\parskip}{10pt}\begin{enumerate} \setlength{\topsep}{10pt}%
\setlength{\itemsep}{15pt}\setlength{\parsep}{5pt}}

\newcommand{\vy}{\ve_y}
\newcommand{\vx}{\ve_x}

\newcommand{\Px}{\mathcal P(Y)}

\newcommand{\uc}{{U^c}}

\newcommand{\vc}{{V^c}}
\newcommand{\wc}{{W^c}}

\newcommand{\ext}{\operatorname*{ext}}

\newcommand{\nxk}{N}

\date{}  %  {\today}

\begin{document}
\title{Jensen measures in potential theory}
\author{Wolfhard Hansen and Ivan Netuka
\thanks{Research supported in part by the project MSM
0021620839 financed by  MSMT and by the grant 201/07/0388
of the Grant Agency of the Czech Republic.}}
\maketitle

\begin{abstract}
It is shown that,  for open sets in classical potential theory and -- more generally -- for elliptic 
harmonic spaces, the set of Jensen measures for a point is a~simple union
of closed faces of a compact convex set  which has been thoroughly studied
a long time ago. In particular, the set of extreme Jensen measures can be 
immediately identified. 
The results hold even without ellipticity (thus capturing also many examples for the heat equation) provided 
a~rather weak  approximation property  for superharmonic functions or a certain transience property holds.
\end{abstract}

Starting with \cite{cole-ransford-subharmonicity,cole-ransford-jensen} several papers have been written in the last
years, where Jensen measures play a central role \cite{roy, poletsky97, khabibullin, alakhrass-thesis,
perkins-thesis, ransford}.
The purpose of this  note is to stress that (at least in classical potential theory                   
and -- more generally --  for  elliptic $\mathcal P$-harmonic spaces~$Y$) the set   of Jensen measures 
(for a point $x$ with respect to~$Y$) is a simple union of 
closed faces of the thoroughly studied compact convex set  of representing measures (for $x$
with respect to the cone of potentials on~$Y$).
In~particular, the set of extreme Jensen measures can
be immediately identified using results which are known even for general balayage spaces since 25 years
(Theorem~\ref{main} and its consequences Corollary~\ref{ap-corollary} and Theorem~\ref{main-ho}). 

The statements hold even without ellipticity (thus capturing also many examples for the heat equation) provided
$Y$ satisfies some (rather weak) approximation property (AP) for superharmonic functions or a  transience property (TP).
In classical potential theory, (TP) is  not only a sufficient, but also a necessary condition for the property that every probability measure, which 
has compact support in~$Y$ and is a representing measure for~$x$ with respect to potentials on~$Y$, is a Jensen measure for~$x$
(Theorem \ref{class-infty}).  Further, for  \emph{bounded} open sets~$Y$ in~$\reald$, this property holds if and only if
every connected component of the boundary of~$Y$ contains a point in the closure of the set of regular points of $Y$
(Corollary \ref{equiv}). 

For  Jensen measures with respect to compact sets see the equalities (\ref{JK}) and~(\ref{JKe}).

To work in reasonable generality, let $Y$ be a $\mathcal P$-harmonic Bauer space (which covers
also the heat equation).
Let $\Px$ be the set of all continuous real potentials on $Y$
and, for every open set $U$ in $Y$, let $\S(U)$ denote the set of all
superharmonic functions on $U$. We recall that $\S^+(Y)$ is 
the set of all superharmonic limits of increasing sequences in $\Px$.

Given a Borel set $E$ in $Y$, let $\C(E)$ be the set of all continuous real functions on~$E$.
For every convex cone $\F$ of lower semicontinuous  functions $f\colon E\to (-\infty,\infty]$ and $x\in Y$,
let~$M_x(\F)$ denote the set of all positive Radon measures~$\mu$ on~$Y$ which are supported
by~$E$ such that, for every $f\in \F$, $\mu(f^-)<\infty$ and $\mu(f)\le f(x)$\footnote{We 
 write $\mu(f)$ instead of $\int f\,d\mu$.}, and let $\ext M_x(\F)$
denote the set of all extreme points of the convex set~$M_x(\F)$.  We recall that
$$
        \ext M_x(\Px)=\{\vx^A\colon A\subset Y\}=\{\vx\}\cup \{\vx^A\colon A\mbox{ Borel  set}, \ x\notin A\}
$$
(see \cite{moko-ext}, \cite[VI.2.2,VI.12.4-5]{BH}). Moreover, if $A$ is a (Borel) set in $Y$ such that $x\notin A$
and $\ve_x^A\ne \vx$, then, by \cite[VI.2.2, 4.1-4.4]{BH},
 there exists a finely closed $G_\delta$-set~$F$ such that 
\begin{equation}\label{AF}
A\subset F\subset \ov A\setminus \{x\} \und \vx^A=\vx^F.
\end{equation} 
Thus
 \begin{equation}\label{ext-moko}
       \ext M_x(\Px)=\{\vx\}\cup \{\vx^F\colon F\mbox{ finely closed $G_\delta$-set},\ x\notin F\}.
\end{equation}  

Given $x\in Y$ and an open neighborhood $U$ of $x$, let $J_x(U)$ denote the set of all
 \emph{Jensen measures} for $x$ with respect to $U$, that is, 
\begin{equation}\label{def-jensen}
 J_x(U) :=\{\mu\in M_x(\S(U))\colon \supp \mu\subset \subset U\}.\footnote{For sets $A\subset B\subset Y$, we write $A\subset \subset B$,
if the closure $\ov A$ of $A$ is a compact subset of $B$.}
\end{equation} 
For example, for every open set $V$ such that $x\in V\subset\subset U$, the harmonic measure $\mu_x^V:=\vx^\vc$
is contained in $J_x(U)$. 
We observe that  $\S(U)$ in (\ref{def-jensen})  can be replaced by $\S(U)\cap \C(U)$,
since every $u\in \S(U)$ is the limit of an increasing sequence in \hbox{$\S(U)\cap \C(U)$}.
If constants are harmonic, then, of course, every $\mu\in J_x(U)$ is a~probability measure. 

Let us note that our assumptions do not exclude that $Y$ is compact. An example is
given by the operator $u''=u$ on the circle $\{z\in \realt\colon |z|=1\}$, where the derivatives
are taken with respect to arc length.

\begin{proposition} If $Y$ is compact, then $J_x(Y)=M_x(\Px)$ and 
$$ \ext J_x(Y)=\{\ve_x\}\cup  \{\ve_x^\vc\colon x \in V\subset Y, \  V\mbox{ \rm finely open $K_\sigma$-set}\} .
$$
\end{proposition} 

\begin{proof} Let $Y$ be compact. Then $\S(U)=\S^+(U)$ by the minimum principle
(and the constant function $0$ is the only harmonic function on $Y$). Hence $J_x(Y)=M_x(\Px)$,
and the proof is finished by (\ref{ext-moko}).
\end{proof} 

In the following, we assume that $Y$ is not compact. Given an open set \hbox{$U\subset \subset Y$}, let
$$
                                   S(U):=\{u\in \C(\ov U)\colon u|_U\in \S(U)\}.
$$
 If $W$ is an open  neighborhood of  $\ov U$, 
then $(\S(W)\cap \C(W))|_{\ov U}\subset S(U)$. 
So, for $x\in Y$, 
\begin{equation}\label{mju}
\bigcup_{x\in U\subset\subset Y} M_x(S(U))=\bigcup_{x\in U\subset\subset Y} J_x(U)\subset J_x(Y)\subset M_x(\Px).
\end{equation}

\begin{theorem}\label{main}
Let $x\in Y$ and suppose that  the union of all $J_x(U)$, $U\subset \subset Y$, is the set~$J_x(Y)$ . Then
\footnote{If $Y$ is a $\mathcal P$-harmonic Brelot space satisfying the axiom of domination, then we may replace ``V finely open''
by ``V fine domain'' (see \cite[Theorem 12.7]{Fug-fharmonic}).} 
\begin{eqnarray}
    \ext  J_x(Y) &=&\bigcup_{x\in U\subset\subset Y} \ext M_x(S(U))\label{ex-main}\\
                       &=& \{\ve_x\}\cup  \{\ve_x^\vc\colon x \in V\subset \subset Y, \  V\mbox{ \rm finely open $K_\sigma$-set}\} .%
\label{fopen}
\end{eqnarray} 
\end{theorem} 

\begin{proof} 
Let $x\in U\subset \subset Y$. By \cite[VII.9.5]{BH}, $M_x(S(U))$ 
is a~closed face of $M_x(\Px)$. Hence (\ref{mju}) and our assumption  imply (\ref{ex-main}).
Moreover, again by \cite[VII.9.5]{BH},
\begin{equation}\label{su-ext}
         \ext M_x(S(U)) = \{\vx\}\cup \{\vx^A\colon A \mbox{ Borel set}, \ x\notin A, \ \beta(\uc)\subset A\},
\end{equation} 
where $\beta(\uc)$ is the largest finely closed subset of $\uc$ such that $\uc\setminus \beta(\uc)$ is semipolar.
To prove (\ref{fopen}) it is sufficient to know that $\beta(\uc)$ is a subset of $\uc$ containing the interior of $\uc$.

Indeed, given a finely open set $V$ such that $x\in V\subset \subset Y$, 
we may choose an open set $U\subset\subset Y$   such that $\ov V\subset U$,  
and then $\vx^\vc\in \ext M_x(S(U))$, by~(\ref{su-ext}).  Conversely, let $U$ be open, $x\in U\subset\subset Y$, and $\mu\in \ext M_x(S(U)) \setminus \{\vx\}$. Then, by~(\ref{su-ext}) and (\ref{AF}),  there exists  a finely open  $K_\sigma$-set $V$ such that $x\in V\subset \ov U$ and~$\mu=\vx^\vc$. 
\end{proof} 

For  a first application of Theorem \ref{main}, we introduce the  following  approximation property:
\begin{description}
\item[(AP)] For every compact $K$ in $Y$, there exists a bounded\footnote{We 
say that a subset of $Y$ is bounded, if it is relatively compact, unbounded, if not.} 
 open neighborhood ~$U$ of~$K$  such that, for every   $u\in \S(U)\cap \C(U)$, there is a function  $v\in \S(Y)\cap \C(Y)$ 
satisfying  $|u-v|<\ve$ on $K$.
\end{description}

If $F$ is a closed set   in  $Y$, then
the connected components of $F^c$ are open, since $Y$ is locally connected. If $K$ is a compact in $Y$, then
the union $\hat K$ of $K$ and all bounded connected components of $K^c$ is compact
(see \cite[Lemma 1]{gardiner-gowri}; its proof uses only that $Y$ is locally compact and locally connected).

\begin{proposition}\label{ell-ap}
If $Y$ is elliptic, then {\rm(AP)} holds.
\end{proposition} 
 
\begin{proof} Let $K$ be a compact set in $Y$, let $L$ be a compact neighborhood of $\hat K$, and let $U$
be the interior of the compact set $\hat L$. 
Since $Y\setminus \hat L$ is the union of all unbounded connected components
of $L^c$, the set $(Y\setminus \hat L)\cup \{\infty\}$ is connected in the Aleksandrov compactification
$Y_\infty:=Y\cup \{\infty\}$, and hence its closure, that is, the set $(Y\setminus U)\cup \{\infty\}$, is connected
as well. Therefore every connected component of $U^c$ is unbounded.

By \cite[Theorem 6.1 and Remark 6.2.1]{BH-simplicial-II}, we obtain that, given $u\in \S(U)\cap \C(U)$ and $\ve>0$,
there exists $v\in \S(Y)\cap \C(Y)$ such that $|u-v|<\ve$ on $\hat K$ (cf.\ also \hbox{\cite[Theorem 6.9]{gardiner-app}} for
the classical case and \cite[Theorem~1]{{gardiner-gowri}} 
for the case of a  Brelot space satisfying the axiom of domination).
\end{proof} 

Moreover, by \cite[Theorem 6.1]{BH-simplicial-II}, it is clear that (AP) holds for many open sets $Y$
in $\real^d\times \real$, $d\ge 1$, with respect to the heat equation. For example, (AP) holds
if $Y$~is a~convex open set in $\reald\times \real$.

By Theorem \ref{main}, we now obtain the following.

\begin{corollary}\label{ap-corollary}
Suppose that $Y$ is elliptic or, more generally, that {\rm(AP)} holds.   
Then, for every $x\in Y$,  $J_x(Y)$ is  the union of all~$J_x(U)$, $U\subset \subset Y$,
and $(\ref{ex-main})$, $(\ref{fopen})$ hold.
\end{corollary} 

\begin{proof} Let $x\in Y$, $\mu\in J_x(Y)$, and $K:=\{x\}\cup \supp \mu$. We choose an open set $U\subset\subset Y$ according to (AP). 
 Let $u\in \S(U)\cap \C(U)$ and $\ve>0$. By (AP), there exists $v\in \S(Y)\cap \C(Y)$ such that $|u-v|<\ve$ on $K$, and therefore
 $$
            \mu(u)\le \mu(v)+ \mu(K)\ve \le v(x)+ \mu(K)\ve\le u(x)+(\mu(K)+1)\ve.
 $$
 Hence $\mu(u)\le u(x)$, $\mu\in J_x(U)$, and the proof is finished by Theorem \ref{main}.
\end{proof}

For the heat equation on an open set $Y$ in $\real^{d+1}$, the union 
$\bigcup_{U\subset \subset Y} J_x(U)$ may be a proper subset of $J_x(Y)$.

\begin{example}
Let $Y:=\{(y',t)\in \real^2\colon t<1+\cos \pi y'\}$ be equipped with the sheaf of solutions to
the heat equation $\partial^2 u/\partial (y')^2=\partial u/\partial t$ and let
 $$
V:=(-2,2)\times (-1,0),  \quad x:=(0,0), \quad \mu:={}^Y\!\ve_x^\vc.
$$
Then 
$$
     \mu\in J_x(Y)\setminus \bigcup_{U\subset\subset Y} J_x(U).
$$
\end{example}

\begin{proof} Indeed,  $\supp \mu$ is the compact set $\partial V\setminus\bigl( (-2,2)\times \{0\}\bigr) $ in $Y$ and $\mu(1)=1$.
For every $n\in\nat$, let 
$$
        V_n:=\{(y',t)\in V\colon t<-1/n\} \und x_n:=(0,-2/n).
$$
Then, for every $u\in \S(Y)\cap \C(Y)$,
$$
            \mu(u)=%\limn {}^Y\!\ve_{x_n}^{V_n^c}(u)\le\limn u(x_n)=u(x).
\limn \mu_{x_n}^{V_n}(u)\le\limn u(x_n)=u(x).
$$
Therefore $\mu\in J_x(Y)$.
 Now let $U$ be an open set such that $x\in U\subset \subset Y$ and $\supp \mu\subset U$. 
There exists $\eta>0$ such that the line segment $\{1\}\times (-\eta,0)$ does not intersect $U$. For $y=(y',t)\in U$,
let $u(y):=0$, if $t\le -\eta$ or $y'<1$, and $u(y):=t+\eta$, if $t>-\eta$ and $y'>1$.
Then $u\in \S(U)$ and  $\mu(u)>0=u(x)$. Hence $\mu\notin J_x(U)$.
\end{proof}

For a second  application of Theorem \ref{main}, which yields an even stronger result,
we introduce the following transience property: %simple property for reduction on compact sets.
\begin{description}
\item[(GTP) ] There is a strictly positive harmonic function $h_0$ on $Y$  such that,
for every compact~$K$ in~$Y$, the set                    $\{R_{h_0}^K=h_0\} $ is compact.
\end{description}
We say that (TP) holds if (GTP) is satisfied with $h_0=1$.

To get a first feeling for (GTP) let us assume for the moment that $h_0>0$ is a harmonic function on $Y$.
Then (GTP) certainly holds if
\begin{equation}\label{ho-infty}
\limsup\nolimits_{y\to\infty}  R_{h_0}^K(y)/h_0(y)<1,
\end{equation} 
which in turn is true,  if there exists a  potential on~$Y$ such that $p>0$ and $p/h_0$
vanishes at infinity. 
In particular, in the classical case,  (TP)  holds,  if $Y$ is a bounded regular set 
in $\reald$, $d\ge 1$ (take any strictly positive continuous potential $p_0$ on an open ball containing the closure
of $Y$ and define $p(y):= p_0(y)- \mu_y^{Y}(p_0)$).
%
% For further sufficient or necessary conditions see In the  case $d\ge 2$, 
% a~necessary and sufficient condition  will be given in Proposition~\ref{class-infty}.

Moreover, let $K$ be a compact in $Y$. 
Then the function $R_{h_0}^K$ is harmonic on~$K^c$, and $R_{h_0}^K=h_0$ on~$\hat K$,
by the minimum principle. Since $\hat K$ is compact, we see that the set $\{R_{h_0}^K=h_0\}$ is compact, if  $R_{h_0}^K<h_0$
on every unbounded connected component~$W$ of~$K^c$. In particular, (GTP) holds, if $Y$ is elliptic and if, for every compact~$K$ in~$Y$, 
there is only one unbounded connected component $W$ of~$K^c$. Indeed, otherwise $R_{h_0}^K=h_0$ on $W$, and hence $R_{h_0}^K=h_0$
on~$Y$, which is impossible, since $R_{h_0}^K$ is bounded by a potential.

\begin{theorem}\label{main-ho}
Suppose that {\rm(GTP)} holds {\rm (}with $h_0${\rm)}. Then, for every $x\in Y$,
\begin{equation}\label{mxk}
         \bigcup_{U\subset \subset Y}J_x(U)=J_x(Y)=\{\mu\in M_x(\Px)\colon \supp \mu\subset\subset Y,\ \mu(h_0)=h_0(x)\}.%
\end{equation} 
Moreover, {\rm (\ref{ex-main})} and {\rm (\ref{fopen})} hold.
\end{theorem} 

\begin{proof} We  fix 
%a harmonic function $h_0$ according to (TP), 
a point $x\in Y$,  a compact~$K$ in~$Y$, and define
$$
        \nxk:= \{\mu\in \M_x(\Px)\colon \supp \mu\subset K,\  \mu(h_0)= h_0(x)\}.
$$
By assumption, there exists a bounded  open neighborhood $W$ of $K\cup \{x\}$
such that $            R_{h_0}^K<h_0$ on $\wc$.
Let $U$ be a bounded open neighborhood of $\ov W$. If we show  that
\begin{equation}\label{nxw}
\nxk\subset M_x(S(U)),
\end{equation} 
then we obtain, by (\ref{mju}), that (\ref{mxk}) holds, and the proof is finished by Theorem \ref{main}.

To prove (\ref{nxw}) we observe that  $\nxk$ is a closed face of $M_x(\Px)$. Hence
it suffices to show that
\begin{equation}\label{extnx}
    \nxk\cap \ext M_x(\Px)\subset M_x(S(U)). 
\end{equation} 
So let us fix $\mu\in \ext M_x(\Px)$, $\mu\ne \vx$, such that $\supp \mu\subset K$ and $\mu( h_0)=h_0(x)$.
By (\ref{ext-moko}), there is a~finely closed $G_\delta$-set $F$ such that $x\notin F$ and $\mu=\vx^F$.
 By \cite[VI.9.4]{BH}, 
$$
          \vx^{F\cap K}=\vx^F|_K+(\vx^F|_{K^c})^{F\cap K}=\vx^F.
$$
So we may assume  that $F\subset K$.
The set $           V:=W\setminus F$ is a~finely open $K_\sigma$-set, $x\in V$, and $\ov V\subset U$. We define
 $\sigma:=1_F\vx^\vc$, $\tau:=1_{\wc}\vx^\vc$, and  note that
$$
                  \vx^\vc=\sigma+\tau \und \mu=\vx^F=\sigma+\tau^F,
$$
where the last equality follows by  \cite[VI.9.4]{BH}.  If $\tau\ne 0$, then

$$
            \tau^F(h_0)=\tau( R_{h_0}^F)\le \tau( R_{h_0}^K)<\tau(h_0),
$$
and therefore
$$
          h_0(x)=\mu(h_0) =(\sigma+\tau^F)(h_0)< (\sigma+\tau)(h_0)=\vx^\vc(h_0),
$$
a contradiction. Thus $\tau=0$ and $\mu=\sigma=\vx^\vc\in \M_x(S(U))$.
\end{proof} 

Given $y, z\in \reald$, let  
$$
u_z(y):=\begin{cases} 
\ln 1/|y-z|,&\mbox{if }d\ge 2,\\
|y-z|^{2-d},&\mbox{if }d\ge 3.
\end{cases}
$$
For classical potential theory, we have the following strong converse
to Theorem \ref{main-ho}.

\begin{theorem}\label{class-infty}
Let  $Y$ be a Greenian open set in $\reald$, $d\ge 1$. Then {\rm(TP)} holds, if
\begin{equation}\label{jxm}
J_x(Y)=\{\mu\in M_x(\Px)\colon \supp\mu\subset \subset Y, \ \mu(1)=1\} \qquad (x\in Y).
\end{equation} 
\end{theorem} 

\begin{proof} We may assume without loss of generality that $Y$ is connected.

If $d=1$, then $Y$ is of the form $(a,b)$, $(a,\infty)$, or $(-\infty,a)$. In the first case,
(TP) holds, since $Y$ is a bounded regular set. In the second case,   every  $p\in\Px$
is increasing, whereas the function $z\mapsto -z$ is harmonic on $Y$, and therefore $\vy\in M_x(\Px)\setminus J_x(Y)$,
whenever $a<y<x<\infty$. Similarly, if $Y=(-\infty, a)$.

So let $d\ge 2$ and  suppose that  (TP) does not hold.  Then $Y\ne \reald$, and there exist a~compact set $K$ in $Y$
and $x_n\in Y\setminus K$ such that $R_1^K(x_n)=1$, for every $n\in\nat$, and the sequence $(x_n)$
is unbounded in $Y$.  So, for every $n\in\nat$, 
$$
\mu_n:=\ve_{x_n}^K\in M_{x_n}(\Px), \quad \supp \mu_n\subset K,\quad \mu_n(1)=1.
$$
Let us  assume first that $\sup_{n\in\nat} |x_n|=\infty$.  Fixing any $z\in\reald\setminus Y$, we may choose $n\in\nat$
such that $u_z(x_n)<u_z$ on $K$. Then  $u_z(x_n)<\mu_n(u_z)$, and hence  $\mu_n\notin J_{x_n}(Y)$.
Next, we suppose that $\sup_{n\in\nat} |x_n|<\infty $. Then $(x_n)$ has a limit point $z\in\partial Y$, and we may 
choose $n\in\nat$ such that  $u_z(x_n)>u_z$ on $K$. Now  $-u_z(x_n)<\mu_n(-u_z)$, and again  $\mu_n\notin J_{x_n}(Y)$.

Thus in both cases (\ref{jxm}) does not hold. This finishes the proof.
\end{proof} 

We introduce the following weak regularity property for open sets $Y$ in $\reald$, $d\ge 1$.
\begin{description}
\item[(WRP)] Every  connected component of $\partial Y$ contains a point in  the closure of~$Y_r$.
\end{description}
Of course, (WRP) trivially holds, if $d=1$. Suppose for a moment that  $d=2$. Then, for every irregular point $z\in\partial Y$ and every 
$\ve>0$, there exists $\delta\in (0,\ve)$, such that the boundary of $B(z,\delta)$ does not intersect $\partial Y$ (see
 \cite[Theorem 7.3.9]{gardiner-armitage}), and hence the singleton $\{z\}$ is the connected component of $\partial Y$ containing $z$.
So (WRP) holds if and only if  $Y_r$ is dense in~$\partial Y$.

Moreover, the following example may be instructive. Let $Y_0$ be the unit ball in~$\reald$, $d\ge 3$, and  let  $I$ be a closed line segment,
$0\in I\subset Y_0$. Choosing pairwise disjoint closed balls $B_n$ in $B(0,1/n)\setminus I$, $n\in\nat$,
such that their union~$T$ is thin at~$0$, we define  $Y:=Y_0\setminus (I\cup T)$. Then $\partial Y\setminus Y_r=I$
and (WRP) holds.

\begin{proposition}\label{regular}
Let $Y$ be an open set in $\reald$, $d\ge 2$, bounded if $d=2$, such that {\rm(WRP)} is satisfied.
Then {\rm(TP)} holds.
\end{proposition} 

\begin{proof} If $d\ge 3$, let $X:=\reald$ and
$p:=|\cdot|^{2-d}$. If $d=2$, we take $R>0$ such that $\ov Y$ is contained in the open disc $X:=B(0,R)$,
and define $p:= \ln R/|\cdot|$. 

We fix a compact~$K$ in~$Y$ and claim that $R_1^K<1$ outside the compact set $\hat K$. 
So let $W$ be  a connected component of~$Y\setminus K$ which is unbounded in~$Y$.
By ellipticity, it suffices to show that $R_1^K(y) < 1$ for some point $y\in W$.
Clearly, this is true, if $W$~is unbounded in~$X$,  since  $R_1^K$ is bounded by a multiple of $p$.

So let us assume that $W$ is bounded in $X$. Then $\emptyset\ne A:=\partial W \setminus K \subset \partial Y\setminus K$. 
We claim that  $A$ intersects the closure $F$ of $Y_r$. Indeed, suppose that $A\cap F=\emptyset$.
Let $z\in A$ and let $B$ be an open ball in $X\setminus (K\cup F)$ containing $z$.
Then  the  set $B\cap \partial Y$ is polar, and
hence $B\setminus \partial Y$ is connected  (see \cite[Corollary 5.1.5]{gardiner-armitage}). 
This implies that $B\setminus \partial Y\subset W$ and  $B\cap \partial Y\subset A$.
So $A$ is open in $\partial Y$. Let $C$ be a~connected component of $\partial Y$ intersecting $A$.
Then $A\cap C$ is both open and closed in~$C$, and hence $A\cap C=C$. By (WRP), $C\cap F\ne \emptyset$.
Hence $A\cap F\ne \emptyset$, proving our claim.

So there exists a point $z\in \partial W\cap F$. Let $B$ be an open ball,  $z\in B\subset\subset  X\setminus K$.
Then  $B\cap Y_r\ne\emptyset$,
and therefore $B\setminus Y$ is nonpolar. We  fix  $y\in B\cap W$. By~\cite[VI.9.4]{BH} or by the minimum principle,
$$
% {}^X\!\vy^{(Y\setminus K)^c}(B)\ge  {}^X\!\vy^{(Y\cap B)^c}(B)>0.
 \mu_y^{Y\setminus K}(B)\ge \mu_y^{Y\cap B}(B)>0.
$$
 By~\cite[VI.2.9]{BH}, we finally conclude that 
$$
% {}^X\!\vy^\wc(B)={}^X\!\vy^{K\cup Y^c}(B)>0 \und R_1^K(y)={}^Y\!\vy^K(K)=  {}^X\!\vy^{K\cup Y^c}(K)<1.
R_1^K(y)={}^Y\!\vy^K(K)=\mu_y^{Y\setminus K}(K)\le 1-\mu_y^{Y\setminus K}(B)<1.
$$
\end{proof}

\begin{remark}\label{remark}{\rm
If $d=2$, the boundedness of $Y$ in Proposition \ref{regular} may not be dropped. Indeed, let
 $Y:=\{y\in\real^2\colon |y|>1\}$.  Then $Y_r=\partial Y$, and hence (WRP) holds.
Choosing $K:=\{y\in \real^2\colon |y|=2\}$,  the set
$\{R_1^K=1\}=\{y\in\realt\colon |y|\ge 2\}$ is not compact (and hence (TP) does not hold). 
Moreover, taking $x:=(3,0)$ and $\mu:={}^Y\!\vx^K$, we have 
$\mu\in M_x(\Px)$, $\supp \mu=K$, and $\mu(1)=1$, but $\mu\notin J_x(Y)$, since $\mu(u_0)>u_0(x)$.
So even (\ref{jxm}) does not hold.
}
\end{remark}

\begin{proposition}\label{converse-plane} 
Let $Y$ be a bounded open set in $\reald$, $d\ge 2$,  such that 
{\rm (WRP)} does not hold. 
Then  {\rm (\ref{jxm})}    does not hold.
\end{proposition} 

\begin{proof} 
By assumption, there exists a connected component $C$ of $\partial Y$ and a connected open neighborhood $U$ of $C$ which  does not intersect $Y_r$.
Then the  set $U\cap \partial Y$ is polar, and hence   $U\setminus \partial Y$
is connected (see \cite[Corollary 5.1.5]{gardiner-armitage}). This implies that $U\setminus \partial Y=U\cap Y$. 
Let $V$ be an open set such that $C \subset V$ and $\ov V\subset U$.
We fix $z\in C$ and  $x\in V\cap Y$ such that $|x-z|<\dist(z,\partial V)$,  and define $\nu:=\mu_x^V$. 
Then $\nu\in M_x(\Px)$,  since polar sets are removable singularities for functions in~$\Px$.
However,  $-u_z>-u_z(x)$ on $\partial V$, and hence  $\nu(-u_z)>-u_z(x)$, $\nu\notin J_x(Y)$. 
\end{proof} 

Combining Theorem \ref{main-ho} with  Propositions \ref{regular} and   \ref{converse-plane} we obtain
the following result. 

\begin{corollary}\label{equiv}
For every  bounded open set $Y$ in $\reald$, $d\ge 2$,
the following statements are equivalent:
\begin{itemize}
\item {\rm (WRP)} holds.
\item {\rm (TP)} holds.
\item   $\bigcup_{U\subset \subset Y}J_x(U)=J_x(Y)=\{\mu\in M_x(\Px)\colon \supp \mu\subset\subset Y,\ \mu(1)=1\}.$
\item  $J_x(Y)=\{\mu\in M_x(\Px)\colon \supp \mu\subset\subset Y,\ \mu(1)=1\}.$
\end{itemize} 
\end{corollary} 

\begin{remark}{\rm
Again, the boundedness of $Y$ cannot be dropped. For $d=2$ see the example discussed in Remark \ref{remark}, where
only (WRP) holds. Suppose now that  $d\ge 3$, let $z_0:=(1,0,\dots,0)$,  and $Y:=\real^d\setminus \real z_0$. Then $\partial Y=\real z_0$
and $Y_r=\emptyset$ so that (WRP) is  not satisfied. But (TP) and the other two properties hold, since $y\mapsto |y|^{2-d}$ is a potential 
on $Y$. However,  \emph{all} properties are satisfied, if $Y:=B(0,1)\setminus \real z_0$.
}
\end{remark}

Finally, let $K$ be a compact set in (a general $\mathcal P$-harmonic Bauer space ) $Y$ and $x\in K$. Then it is, of course, possible to define
$$
           J_x(K):=\bigcap_{K\subset U } J_x(U)
                 $$
(see \cite{perkins-thesis}). However,  this does not yield anything new (see \cite[Sections VII.8 and VII.9]{BH}). Indeed, clearly
$$
            S_0(K):=\bigcup_{K\subset U}S(U)|_K =\bigcup_{K\subset U} (\S(U)\cap \C(U))|_K.  
$$
Let $G$ be the fine interior of $K$. Then, by \cite[VII.9.2]{BH}, the uniform closure of~$S_0(K)$ is the set~$S(K,G)$ of all continuous real functions~$u$ on~$K$ 
such that $\ve_y^\vc(u)\le u(y)$, for all finely open~$V\subset \subset G$ and all $y\in V$. Thus
\begin{equation}\label{JK}
        J_x(K)=M_x(S_0(K))=M_x(S(K,G)).
\end{equation} 
By \cite[VII.9.5]{BH}, $M_x(S(K,G))$ is a closed face of $M_x(\Px)$ and
\begin{equation}\label{JKe}
     \ext M_x(S(K,G))=\{\vx\}\cup \{\vx^{B^c}\colon B\mbox{ Borel},\ x\in B\subset K\}.
\end{equation} 

%\bibliography{../../../../Forschung/lit_bank}
\bibliographystyle{plain}

{\small \noindent 
Wolfhard Hansen,
Fakult\"at f\"ur Mathematik,
Universit\"at Bielefeld,
33501 Bielefeld, Germany, e-mail:
 hansen$@$math.uni-bielefeld.de}\\
{\small \noindent Ivan Netuka,
Charles University,
Faculty of Mathematics and Physics,
Mathematical Institute,
 Sokolovsk\'a 83,
 186 75 Praha 8, Czech Republic, email:
netuka@karlin.mff.cuni.cz}

\end{document}